\documentclass[reqno,12pt]{amsart}%
\usepackage{amsfonts}
\usepackage{amsmath}
\usepackage{amssymb}

\def\bE{{\mathbb{E}}}

\def\bR{{\mathbb{R}}}

\def\cC{{\mathcal{C}}}

\newcommand\bld[1]{\boldsymbol{#1}}

\def\dWS{\mathrm{d}_{{w}}}

\newtheorem{theorem}{Theorem}
\theoremstyle{plain}

\newtheorem{proposition}[theorem]{Proposition}

\newcommand\bel[1]{\begin{equation}\label{#1}}
\newcommand\ee{\end{equation}}

\numberwithin{theorem}{section}
\numberwithin{equation}{section}

\makeatletter
\@namedef{subjclassname@2020}{\textup{2020} Mathematics Subject Classification}
\makeatother

\begin{document}

\title[Gaussian Functional CLT]
{A Sharp Rate of Convergence in the Functional Central Limit  Theorem with Gaussian Input}

\author{S. V. Lototsky}
\curraddr[S. V. Lototsky]{Department of Mathematics, USC\\
Los Angeles, CA 90089}
\email[S. V. Lototsky]{lototsky@math.usc.edu}
\urladdr{https://dornsife.usc.edu/sergey-lototsky/}



\begin{abstract}
When the underlying random variables are Gaussian, the classical Central Limit Theorem (CLT) is trivial, but the functional CLT is not.
The objective of the paper is to investigate   the functional CLT for stationary Gaussian processes in the Wasserstein-1 metric 
on the space of continuous functions. Matching upper and lower bounds are established, indicating that the convergence rate is
 slightly faster than  in the L\'{e}vy-Prokhorov metric.
\end{abstract}

\maketitle

\today

\section{Introduction}
By the Central Limit Theorem, given a collection $\{\xi_k,\ k\geq 1\}$ of independent and identically distributed random variables, each with mean zero and variance one,
the sequence $S_n=n^{-1/2}\sum_{k=1}^n \xi_k$ converges in distribution, as $n\to \infty$, to the standard Gaussian random variable.
The Berry-Esseen bound \cite[Theorem 15.51]{Klenke} gives the rate of convergence in the Kolmogorov metric:
\begin{equation}
\label{rate-CLT}
\sup_{x\in \bR} |F_n(x)-\Phi(x)|\leq \frac{\bE|\xi_1|^3}{\sqrt{n}},
\end{equation}
where $F_n$ is the cumulative distribution function of $S_n$ and $\Phi$ is  the cumulative distribution function of the standard Gaussian random variable.
While the  rate $1/\sqrt{n}$ is sharp in general, it can be improved by imposing additional conditions on the random variables $\xi_k$.
For example, if $\bE \xi_k^3=0$ and $\bE \xi_k^4<\infty$, then   the left-hand side of \eqref{rate-CLT} is of order $1/n$; cf \cite[Theorem 5.2.1]{Petrov-LT}.
Of course, if each  $\xi_k$ is standard normal, then the left-hand side of \eqref{rate-CLT} is zero.

The functional version of the Central Limit Theorem, also known as the Donsker invariance principle \cite[Theorem 21.43]{Klenke}, establishes weak convergence of
the sequence of processes
\bel{Sn}
 S_n(t)=n^{-1/2}\sum_{k=1}^{\lfloor n t \rfloor} \xi_k+ \frac{nt-\lfloor n t \rfloor}{\sqrt{n}} \xi_{\lfloor n t \rfloor+1},\ t\geq 0,\ n\geq  1,
\ee
 to the standard Brownian motion $W$.
 An  analog of  \eqref{rate-CLT} becomes a bound on the distance between the distributions of $S_n$ and $W$ on the space of continuous functions $\cC(0,T)$
  in the L\'{e}vy-Prokhorov metric. Compared to \eqref{rate-CLT}, the
 corresponding rate of convergence depends on integrability properties of $\xi_k$ in a more complicated way:
 if $\bE|\xi_1|^p<\infty$, $p>2$, then the rate $n^{-(p-2)/(2(p+1))}$ is sharp;
 if $\bE e^{t\xi_1}<\infty,\ |t|<\delta, \ \delta>0,$  then the rate $\ln n/\sqrt{n}$ is sharp. For details, see \cite[Chapter 1]{CH-weighted}; earlier works on the
 subject include \cite{KMT1, KMT2, Sakhanenko-rate1}.

 A more general approach to investigating the rate of convergence is to find a  bound on
 \bel{dist-g}
 \sup_{\varphi \in \mathcal{G}} |\bE \varphi(S_n)-\bE \varphi(W)|
 \ee
  for a suitable class $\mathcal{G}$ of functions $\varphi: \cC(0,T)\to \bR$.
 Barbour \cite[Theorem 1]{Barbour-diff-stein} used an infinite-dimensional version of Stein's  method to establish the benchmark result
\begin{equation}
\label{rate-IP}
|\bE \varphi(S_n)-\bE \varphi(W)| \leq \frac{C}{\sqrt{n}}\Big( \sqrt{\ln n} + \bE|\xi_1|^3\Big)
\end{equation}
for a certain (rather restrictive) class $\mathcal{G}$; the restrictive nature of this  class ensures that there is no contradiction with  \cite[Chapter 1]{CH-weighted} or \cite{Sakhanenko-rate1}. For various other $\mathcal{G}$,  there are  bounds of the form
\begin{equation}
\label{rate-IP1}
|\bE \varphi(S_n)-\bE \varphi(W)| \leq \frac{C}{{n}^r},\ 0<r<\frac{1}{2};
\end{equation}
  cf. \cite{Donsker-rate1} and references therein.

 Unlike \eqref{rate-CLT}, the left-hand side of \eqref{rate-IP} will not be zero even if the random variables $\xi_k$ are
Gaussian, as long as $\mathcal{G}$  is rich enough to capture the infinite-dimensional nature of the problem. In fact, for certain   $\mathcal{G}$, one can
use \eqref{dist-g} to  define a metric on the space of distributions. For example, if  $\mathcal{G}$ is the collection of { bounded} Lipschitz continuous functions,
then  convergence in the corresponding { bounded Lipschitz metric}  is equivalent to weak convergence, that is, convergence in the L\'{e}vy-Prokhorov metric; cf.
\cite[Theorem 13.16]{Klenke} or \cite[Theorem 11.3.3]{Dudley-RAP}.
 Removing the boundedness condition (for example, to include linear functionals) leads to the Wasserstein-1 metric, which is the subject of this paper.

 Recall that, for two probability measures $\mu$, $\nu$ on a  complete separable metric space $E$ with distance function $\rho$ and the corresponding
 Borel sigma-algebra $\mathcal{B}(E)$,
 \begin{itemize}
 \item the {\tt bounded Lipschitz metric} is
\bel{BL-d}
\mathrm{d}_{_{BL}}(\mu,\nu)=\sup_{\varphi} \left|  \int_E \varphi d\mu - \int_E \varphi  d\nu\right|,
\ee
with supremum over functions  $\varphi:E \to \mathbb{R}$ such that, for all $x,y\in E$, $|\varphi(x)-\varphi(y)|\leq \rho(x,y)$  and $|\varphi(x)|\leq 1$;
 \item  the {\tt Wasserstein-1 metric} $\dWS(\mu,\nu)$, also known as the {\tt Kantorovich-Ru\-bin\-stein metric},  is
\bel{W-KR-d}
\dWS(\mu,\nu)=\sup_{\varphi} \left|  \int_E \varphi d\mu - \int_E \varphi  d\nu\right|,
\ee
with supremum over functions $\varphi:E \to \mathbb{R}$ such that, for all $x,y\in E$,  $|\varphi(x)-\varphi(y)|\leq \rho(x,y)$;
\item the {\tt L\'{e}vy-Prokhorov metric} is
\bel{LP-d}
\mathrm{d}_{_{LP}}(\mu,\nu)=\inf\{\varepsilon>0: \mu(A)\leq \nu(A^{\varepsilon})+\varepsilon,\ A\in \mathcal{B}(E)\},
\ee
where $A^{\varepsilon}=\{x\in E: \inf\limits_{y\in A} \rho(x,y)\leq \varepsilon\}.$
\end{itemize}
We have $\mathrm{d}_{_{BL}}(\mu,\nu)\leq \dWS(\mu,\nu)$ (by definition),  $\mathrm{d}_{_{LP}}(\mu,\nu)\leq \sqrt{\mathrm{d}_{_{BL}}(\mu,\nu)}$
(\cite[Proof of Theorem 11.3.3]{Dudley-RAP}), and $\mathrm{d}_{_{BL}}(\mu,\nu)\leq 4 \mathrm{d}_{_{LP}}(\mu,\nu)$ (\cite[Corollary 11.6.5]{Dudley-RAP}).
In particular, convergence in the Wasserstein-1 metric implies weak convergence, that is, convergence in either  L\'{e}vy-Prokhorov or bounded Lipschitz metric;
the converse is not always true \cite[p. 421]{Dudley-RAP}; in fact, the diagram in \cite{GibbsSu} suggests that
  the Wasserstein-1 metric $\dWS$ is the strongest possible for the CLT-type problems in function spaces.
  Still, a { sharp} rate of convergence in one  metric does not directly lead to a { sharp}
rate in any other metric.

The invariance principle  can hold if independence requirement for the random variables $\xi_k$ is relaxed, for example,
 to  a strictly stationary and  ergodic martingale difference \cite[Theorem 9.1.1]{LSh.m}, or  a stationary Markov process satisfying Doebling's condition
 \cite{Donsker-rate3} [where a bound of the type \eqref{rate-IP1} is also established].  In  continuous time,
 if $X=X(t),\ t\in \bR,$ is a strictly stationary process with  mean zero and covariance function $R(t)=\bE \big(X(t)X(0)\big)$
 satisfying $\int_{-\infty}^{+\infty} R(t)\, dt =1$,
then, under some additional conditions of weak dependence, the sequence of processes
\bel{CT000}
S_n(t)=\frac{1}{\sqrt{n}}\int_0^{nt} X(s)\, ds,\ n=1,2,\ldots, t\in [0,1],
\ee
converges weakly to the standard Brownian motion; cf. \cite[Theorem 9.2.1]{LSh.m} or \cite[Theorem VIII.3.79]{J-Sh}.

 The objective of this paper  is to
show that if $X$ is a stationary Gauss-Markov process, in either discrete or continuous time, then the
Wasserstein-1 distance between $S_n$ and $W$ in the space of continuous functions is of order $\big(n^{-1}\ln n\big)^{1/2}$.
In other words,  if $S_n$ is  Gaussian,  then  the convergence rate in Wasserstein-1 metric is slightly faster  than the L\'{e}vy-Prokhorov  rate $\ln n/\sqrt{n}$.
This difference   does not contradict the results from \cite[Chapter 1]{CH-weighted} and \cite{Sakhanenko-rate1},
and the discrepancy by a  $\sqrt{\ln n}$ factor can be explained as follows:
 in the limit $\sigma\to 0+$,  the distance between  a  Gaussian distribution with mean zero and variance $\sigma^2$  and a point mass at zero is of order $\sigma$
  in the Wasserstein-1 metric, but it is of order $\sigma\sqrt{|\ln \sigma|}$ in the L\'{e}vy-Prokhorov metric.

 Section \ref{sec:CT} discusses (the easier) continuous-time case \eqref{CT000}.
Discrete-time  case, a generalization of \eqref{Sn} for a stationary Gaussian sequence $\{\xi_k,\ k\geq 0\}$,
 is in Section \ref{sec:DT}. In Section \ref{sec:ex}, the results are applied to  weak approximation for some ordinary differential equations with additive noise.
Section \ref{sec:sum} is a summary. Traditionally,  models \eqref{Sn}   and \eqref{CT000}   are studied on a bounded time interval $[0,T]$, and
the index parameter $n=1,2,\ldots$ is discrete. In this paper, the time interval is   $(0,+\infty)$ and the index parameter $\kappa\geq 1$  is not necessarily  an integer.

The following two properties of Gaussian processes will be used on several occasions: \\

1. {\tt The Borell-TIS inequality} \cite[Theorem 2.1.1]{StochGeometry-AdlerTaylor}: If $X=X(t), \ t\in \mathcal{T},$ is a zero-mean
Gaussian process indexed by the set $\mathcal{T}$, and
$\mathbb{P}\left( \sup\limits_{t\in \mathcal{T}} X(t) < \infty\right) =1$, then $\bE \sup\limits_{t\in \mathcal{T}} |X(t)| < \infty$ and,
with $X^*=\bE\sup\limits_{t\in \mathcal{T}} X(t),\ \
\sigma_X^2=\sup\limits_{t\in \mathcal{T}}\bE X^2(t)$,
\bel{I-B-TIS}
\mathbb{P}\Big(\sup\limits_{t\in \mathcal{T}} X(t) - X^*>x\Big)\leq e^{-x^2/(2\sigma_X^2)},\ x>0.
\ee

2. {\tt The Fernique-Sudakov inequality} \cite[Theorem 2.2.3]{StochGeometry-AdlerTaylor}: If $X=X(t), \ Y=Y(t)\ t\in \mathcal{T},$ are zero-mean
Gaussian processes indexed by the set $\mathcal{T}$, and, for all $t,s\in \mathcal{T}$, $\bE|X(t)-X(s)|^2 \leq \bE |Y(t)-Y(s)|^2$, then
\bel{I-FS}
\bE\sup\limits_{t\in \mathcal{T}} X(t) \leq \bE\sup\limits_{t\in \mathcal{T}} Y(t).
\ee

\section{Continuous Time}
\label{sec:CT}

 Let $X=X(t),\ t\in \bR,$ be a (continuous version of a) stationary Gaussian process with mean zero and covariance
 $$
 \bE X(t)X(s)=e^{-2|t-s|}.
 $$
  In particular, $X(t)$ is a standard Gaussian random variable for every $t$.
Equivalent characterizations of $X$  are as follows:
 \begin{align}
 \label{S-OU0}
 X(t)&=e^{-t}W(e^{2t});\\
 \label{S-OU1}
 X(t)&=2\int_{-\infty}^t e^{-2(t-s)}dW(s);\\
 \label{S-OU2}
 dX(t)&=-2X(t)dt+2\,dW(t),\ t\geq 0.
 \end{align}
 In  \eqref{S-OU0}, \eqref{S-OU1}, and \eqref{S-OU2},  $W=W(t), \ t\geq 0,$ is a standard Brownian motion;  in \eqref{S-OU1}, when $t<0$,
  $W(t)=V(-t)$ for an independent copy $V$ of $W$.
 The initial condition  $X(0)$ in \eqref{S-OU2}  is a standard Gaussian random variable independent of $W$.

 Given a {\em real number} $\kappa\geq 1$, we define
 \begin{equation}
 \label{Wk}
 W^{\kappa}(t)=\frac{1}{\sqrt{\kappa}}\int_0^{\kappa t} X(s)\, ds,\ \ t\geq 0.
 \end{equation}

Denote by $\cC_{(0)}$ the collection of continuous functions $f=f(t)$ on $[0,+\infty)$ such that
$$
f(0)=0,\ \lim_{t\to +\infty} \frac{|f(t)|}{t}=0.
$$
Endowed with  the norm
$$
\|f\|_{(0)}=\sup_{t>0}\frac{|f(t)|}{1+t},
$$
 $\cC_{(0)}$ becomes a separable Banach space; cf. \cite[Section 1.3]{DeuStr-LD}.
We have $\mathbb{P}(W\in \cC_{(0)})=1$  (either by the law of large numbers for square integrable martingales or
using that $t\mapsto tW(1/t)$ is a standard Brownian motion), and also
$\mathbb{P}(W^{\kappa}\in \cC_{(0)})=1$, $\kappa\geq 1$, because the ergodic theorem implies
$$
\lim_{t\to \infty} \frac{1}{t} \int_0^t X(s)\, ds = \bE X(0) = 0
$$
 with probability one.

 \begin{proposition}
 \label{prop-approxW}
 There exists a constant $C_X$ such that, for every $\kappa\geq 1$ and every  function $\varphi:  \cC_{(0)}\to \bR$ satisfying
 \begin{equation}
 \label{Lip1}
 |\varphi(f)-\varphi(g)|\leq \|f-g\|_{(0)},
 \end{equation}
 we have
 \begin{equation}
 \label{WD-Wk}
 \Big|\bE\varphi\big(W^{\kappa}\big)-\bE\varphi(W)\Big|\leq C_X\left(\frac{\ln(1+\kappa)}{\kappa}\right)^{1/2}.
 \end{equation}
\end{proposition}

\begin{proof}
Using \eqref{S-OU2},
$$
X(t)=X(0)e^{-2t}+2\int_0^t e^{-2(t-s)}\, dW(s).
$$
 Changing the order of integration (stochastic Fubini theorem \cite[Theorem IV.46]{Protter}),
\begin{equation}
\label{approxW-1}
W^{\kappa}(t)=\kappa^{-1/2}W(\kappa t)+\frac{X(0)-X(\kappa t)}{2\sqrt{\kappa}}.
\end{equation}
Define
\begin{equation}
\label{aux-not}
V^{\kappa}(t)=\kappa^{-1/2}W(\kappa t),\ \ \ X^{\kappa}(t)=\frac{X(\kappa t)-X(0)}{2}.
\end{equation}
Because $t\mapsto V^{\kappa}(t)$, $t\geq 0$, is standard Brownian motion for every $\kappa>0$, we have
$\bE \varphi\big(W\big)=\bE \varphi\big(V^{\kappa}\big)$ and then, using \eqref{Lip1},
$$
|\bE\varphi\big(W^{\kappa}\big)-\bE\big(\varphi(W)\big)\Big|\leq \frac{\bE\|X^{\kappa}\|_{(0)}}{\sqrt{\kappa}}.
$$
Next, by \cite[Proposition 2.1]{XChen-LIL},
\begin{equation}
\label{OU-sup}
\lim_{T\to +\infty} \frac{\max_{0\leq t\leq T} X(t)}{\sqrt{2\ln T}}=1,
\end{equation}
with probability one.
 Using the same arguments as in \cite[Proof of Proposition 10.2]{GaussProcLect-Lifshits}, we conclude from \eqref{OU-sup} that
\begin{equation}
\label{OU-sup1}
\limsup_{T\to +\infty} \frac{\max_{0\leq t\leq T} |X(t)|}{\sqrt{2\ln T}}\leq 2,
\end{equation}
 and then continuity of $X$ implies that the random variable
$$
\zeta=\sup_{t>0} \frac{|X(t)-X(0)|}{2\sqrt{\ln(2+t)}}
$$
is finite   with probability one. Indeed, if
$$
T^*=\sup\left\{ t\geq 0: \frac{\max_{0\leq t\leq T} |X(t)-X(0)|}{2\sqrt{2\ln(2+ T)}}>5\right\}
$$
then, by \eqref{OU-sup1}, $\mathbb{P}(T^*<\infty)=1$, so that
$$
\zeta \leq \max_{0\leq t\leq T^*} \frac{|X(t)-X(0)|}{2\sqrt{\ln(2+t)}} + 5 <\infty.
$$
Moreover,  because $t\mapsto X(t)-X(0)$ is a Gaussian process with mean zero, the Borell-TIS inequality \eqref{I-B-TIS} implies
 \begin{equation}
 \label{sup-int}
 \bE \zeta^p <\infty
 \end{equation}
 for all $p>0$.

If  $t>0$ and $\kappa\geq 1$, then, by direct computation, 
$$
(2+\kappa t)\leq (1+\kappa)(1+t),\ \frac{1}{2}\leq \sqrt{\ln(1+\kappa)}, \ \frac{\sqrt{\ln(1+t)}}{1+t}\leq \frac{1}{2}.
$$
As a result,
\begin{align*}
\frac{\sqrt{\ln(2+\kappa t)}}{1+t}&\leq \frac{\sqrt{\ln(1+t)}}{1+t} + \sqrt{\ln(1+\kappa)}\leq 2\sqrt{\ln(1+\kappa)},\\
\|X^{\kappa}\|_{(0)} &=
\sup_{t>0} \frac {|X(\kappa t)-X(0)|}{2(1+t)} \\
&\leq \left(\sup_{t,\kappa>0}\frac {|X(\kappa t)-X(0)|}{2\sqrt{\ln(2+\kappa t)}}\right)
\left( \sup_{t>0} \frac{\sqrt{\ln(2+\kappa t)}}{1+t}\right)
\leq 2\zeta\sqrt{\ln(1+\kappa)},
\end{align*}
and  \eqref{WD-Wk} follows with
\begin{equation}
\label{CX}
C_X=\bE\left[\sup_{t>0} \frac{|X(t)-X(0)|}{\sqrt{\ln(2+t)}}\right].
\end{equation}
 \end{proof}

Denote by $\bld{\mu}_0$ and $\bld{\mu}_{\kappa}$ the measures on $\cC_{(0)}$  generated by the processes $W$ and $W^{\kappa}$. The following is the main result of this
section, showing that the convergence rate $\kappa^{-1/2}\sqrt{\ln\kappa}$ is sharp for the Wasserstein-1 metric.

\begin{theorem}
\label{th:rate}
There exist positive constants  $C_X$ and $c_X$ such that, for every $\kappa\geq 1$,
\begin{equation}
\label{eq:rate2}
c_X\left(\frac{\ln(1+\kappa)}{\kappa}\right)^{1/2}\leq \dWS(\bld{\mu}_{\kappa},\bld{\mu}_0)\leq C_X\left(\frac{\ln(1+\kappa)}{\kappa}\right)^{1/2}.
\end{equation}
\end{theorem}

\begin{proof}
The upper bound in \eqref{eq:rate2} follows from \eqref{W-KR-d} and Proposition  \ref{prop-approxW}.
To establish the lower bound, we  use \eqref{W-KR-d} with a particular  $\varphi$.

If  $\varphi: \cC_{(0)}\to \bR$ is a bounded linear functional, then \eqref{approxW-1}
 and \eqref{aux-not}  imply
\begin{equation}
\label{aux-lin}
\bE \varphi(W^{\kappa}) - \bE\varphi(W)=\kappa^{-1/2}\bE\varphi(X^{\kappa}).
\end{equation}
For $f\in \cC_{(0)},$ define
$$
\varphi_{\kappa}: f \mapsto \frac{f(t^*_{\kappa})}{2},
$$
where
$$
t^*_{\kappa} = \arg\max_{0\leq t \leq 1} X^{\kappa}(t).
$$
In particular, $t^*_{\kappa}\in [0,1]$. Then  $\varphi_{\kappa}$
is a  bounded linear functional on $\cC_{(0)}$:
$$
|\varphi_{\kappa}(f)|\leq \frac{\max_{0\leq t\leq 1} |f(t)|}{2}\leq \max_{0\leq t\leq 1}\frac{|f(t)|}{1+t}
\leq \|f\|_{(0)}.
$$
Therefore, by \eqref{W-KR-d} and  \eqref{aux-lin},
\begin{equation}
\label{WS-lb1}
\dWS(\bld{\mu}_{\kappa},\bld{\mu}_0) \geq \kappa^{-1/2}|\bE\varphi_{\kappa}(X^{\kappa})| =\frac{\bE\max_{0\leq t \leq 1} X^{\kappa}(t)}{2\sqrt{\kappa}}.
\end{equation}

Next, define
$$
\zeta_{\kappa}=\frac{\max\limits_{0\leq t \leq 1} X^{\kappa}(t)}{\sqrt{\ln(1+\kappa)}}
 =\frac{\max\limits_{0\leq t \leq \kappa}X(t)-X(0)}{2\sqrt{\ln(1+\kappa)}}, \ \ \kappa\geq 1;
$$
the second equality follows from  \eqref{aux-not}.
By \eqref{sup-int}, the family $\{\zeta_{\kappa},\ \kappa\geq 1\}$ is uniformly integrable, so that  \eqref{OU-sup} implies
\begin{equation}
\label{L1-conv}
\lim_{\kappa\to \infty} \bE \zeta_{\kappa} = \frac{1}{\sqrt{2}},
\end{equation}
which, in turn,  means
\begin{equation}
\label{L1-conv1}
\inf_{\kappa\geq 1} \bE \zeta_{\kappa}>0.
\end{equation}
The lower bound in \eqref{eq:rate2}, with
\bel{cX}
c_X=\frac{1}{2}\inf_{\kappa\geq 1} \frac{\bE\left[ \max\limits_{0\leq t \leq \kappa} X(t)\right]}{\sqrt{\ln(1+\kappa)}},
\ee
now follows from   \eqref{WS-lb1} and \eqref{L1-conv1}, because $\bE X(0)=0$.
\end{proof}

To get a better idea about  numerical values of $C_X$ and $c_X$, we need
\begin{proposition}
\label{prop-maxX}
Let $X=X(t),\ t\in \bR,$ be a stationary Gaussian process with mean zero and
covariance $e^{-2|t-s|}$. Define the random variable
\bel{eta-X}
\bar{\eta}=\max_{0\leq t\leq 1} X(t).
\ee
Then
\begin{align}
\label{X-FS}
 1.2& <  \bE \bar{\eta} < 3.2;\\
\label{X-BTIS}
\mathbb{P}(\bar{\eta}>x)&\leq e^{-(x-3.2)^2/2},\ x\geq 3.2.
\end{align}
\end{proposition}

\begin{proof}
Let $B=B(t)$, $0\leq t\leq 1$, be the standard Brownian bridge and $W=W(t)$, $0\leq t\leq 1$, a standard Brownian motion. Then
\begin{align*}
\bE|B(t)-B(s)|^2 &= |t-s|-|t-s|^2,\  \bE|X(t)-X(s)|^2 = 2(1-e^{-2|t-s|}),\\
& \bE|W(t)-W(s)|^2=|t-s|,
\end{align*}
and, because
$$
x-x^2\leq 1-e^{-2x}\leq 2x,\ x\geq 0,
$$
 inequality \eqref{I-FS}  implies
$$
2\,\bE \max_{0\leq t\leq 1} B(t) \leq \bE \bar{\eta} \leq 4\, \bE \max_{0\leq t\leq 1} W(t).
$$
It is well known (e.g. \cite[Section 12.3]{Dudley-RAP}) that
\bel{BB-tail}
\mathbb{P}\left(\max_{0\leq t\leq 1} B(t) >x\right) = e^{-2x^2}, \ \
\mathbb{P}\left(\max_{0\leq t\leq 1} W(t) >x\right) = \frac{2}{\sqrt{2\pi}}\int_{x}^{+\infty}e^{-t^2/2}dt.
\ee
Then
\begin{align*}
\bE \max_{0\leq t\leq 1} B(t) &= \int_0^{+\infty} e^{-2x^2}\, dx = \frac{\sqrt{2\pi}}{4}>0.6,\\
  \bE \max_{0\leq t\leq 1} W(t) & =
\frac{2}{\sqrt{2\pi}}\int_0^{+\infty}xe^{-x^2/2}\, dx  = \sqrt{\frac{2}{\pi}}<0.8,
\end{align*}
and \eqref{X-FS} follows. After that, \eqref{X-BTIS} is a re-statement of the Borell-TIS inequality \eqref{I-B-TIS}.
\end{proof}

We can now show that  the number $C_X$ defined in \eqref{CX} satisfies
\bel{CX-X}
1.4 <  C_X < 14.
\ee
For the lower bound, note that
$$
C_X\geq \bE\left[\sup_{t>0} \frac{X(t)-X(0)}{\sqrt{\ln(2+t)}}\right],
$$
whereas
$$
\sup_{t>0} \frac{X(t)-X(0)}{\sqrt{\ln(2+t)}}\geq \limsup_{T\to \infty} \frac{\max\limits_{0\leq t\leq T} \big(X(t)-X(0)\big)}{\sqrt{\ln(2+T)}},
$$
and it remains to apply \eqref{L1-conv}.

For the upper bound in \eqref{CX-X}, start by writing
\bel{CX-00}
C_X\leq \bE\left[\sup_{t>0} \frac{|X(t)|}{\sqrt{\ln(2+t)}}\right] + \frac{\bE|X(0)|}{\sqrt{\ln 2}}; \ \ \  \frac{\bE|X(0)|}{\sqrt{\ln 2}}=\sqrt{\frac{2}{\pi \ln 2}}\approx 0.96.
\ee
Next,  let $\bar{X}$ be the process consisting of iid copies of $X(t),\ t\in [0,1),$ on each of the intervals $[k-1,k)$, $k=1,2,\ldots$.
Then $\bE X^2(t)=\bE \bar{X}^2(t)$, $\bE X(t)X(s)\geq \bE \bar{X}(t)\bar{X}(s)$, and so
\bel{CX-1}
\bE\left[\sup_{t>0} \frac{|X(t)|}{\sqrt{\ln(2+t)}}\right] \leq 2\bE\left[\sup_{t>0} \frac{X(t)}{\sqrt{\ln(2+t)}}\right] \leq 2\bE\left[\sup_{t>0} \frac{\bar{X}(t)}{\sqrt{\ln(2+t)}}\right],
\ee
where the first inequality follows from \cite[Proposition 10.2]{GaussProcLect-Lifshits}, and the second, from   \eqref{I-FS}.
On the other hand, if $\bar{\eta}_k$, $k\geq 1$, are iid copies of the random variable $\bar{\eta}$ from \eqref{eta-X}, then
\bel{CX-2}
\bE\left[\sup_{t>0} \frac{\bar{X}(t)}{\sqrt{\ln(2+t)}}\right]\leq \bE\left[\sup_{k\geq 1} \frac{\bar{\eta}_k}{\sqrt{\ln(1+k)}}\right].
\ee
Using \eqref{X-BTIS} with  $x>5$,
\begin{align*}
\mathbb{P}\Big(\sup_{k\geq 1} \frac{\bar{\eta}_k}{\sqrt{\ln(1+k)}} > x\Big) &\leq \sum_{k\geq 1} \mathbb{P}\Big(\bar{\eta}_k>x\sqrt{\ln(1+k)}\Big)\\
&\leq \sum_{k\geq 1}
\frac{1}{(1+k)^{-(x-3.2)^2/2}}\leq \frac{2}{(x-3.2)^2-2},
\end{align*}
and therefore
$$
\bE\left[\sup_{k\geq 1} \frac{\bar{\eta}_k}{\sqrt{\ln(1+k)}}\right]\leq 5+2\int_5^{+\infty} \frac{dx}{(x-3.2)^2-2} <6.5.
$$
Then \eqref{CX-X} follows from  \eqref{CX-00} -- \eqref{CX-2}.

Next, we will show that  the number $c_X$ defined in \eqref{cX} satisfies
\bel{cx-x}
0.2 < c_X <  0.8.
\ee

Indeed, the upper bound  follows immediately from \eqref{L1-conv}.
For the lower bound, start by noting that, for $N\leq \kappa < N+1$,
$$
\bE\left[ \max\limits_{0\leq t \leq \kappa} X(t)\right] \geq \bE\left[ \max\limits_{k=1,\ldots, N} X(k)\right]
$$
and $\bE|X(k)-X(m)|^2 = 2(1-e^{-2})> 1$. Now take iid Gaussian $Y_k$, $k=1,\ldots, N$ with mean zero and variance $1/2$. Then
$\bE|Y(k)-Y(m)|^2=1$ and so
$$
\bE|X(k)-X(m)|^2\geq \bE|Y(k)-Y(m)|^2.
$$
By  \eqref{I-FS},
$$
\bE\left[ \max\limits_{k=1,\ldots, N} X(k)\right] \geq \bE\left[ \max\limits_{k=1,\ldots, N} Y(k)\right],
$$
and, by \cite[Lemma 10.2]{GaussProcLect-Lifshits},
$$
\bE\left[ \max\limits_{k=1,\ldots, N} Y(k)\right] \geq 0.4\sqrt{\ln N},
$$
leading to the lower bound in \eqref{cx-x}.

To summarize, we can write \eqref{eq:rate2} in a more explicit form
$$
0.2\left(\frac{\ln(1+\kappa)}{\kappa}\right)^{1/2}\leq \dWS(\bld{\mu}_{\kappa},\bld{\mu}_0)\leq 14\left(\frac{\ln(1+\kappa)}{\kappa}\right)^{1/2},\ \kappa\geq 1.
$$

 Theorem \ref{th:rate} can be used to study convergence on a bounded interval.
For $T>0$, let $\cC(0,T)$ be the space of continuous functions on $[0,T]$ with the sup norm, and  denote by $\bld{\mu}_{0}^T$ and $\bld{\mu}_{\kappa}^T$ the measures on $\cC(0,T)$  generated by the processes $W$ and $W^{\kappa}$.

\begin{theorem}
\label{th:rateT}
There exist positive constants  $C_{X,T}$ and $c_{_{X,T}}$ such that, for every $\kappa\geq 1$,
\begin{equation}
\label{eq:rate2T}
c_{_{X,T}}\left(\frac{\ln(1+\kappa)}{\kappa}\right)^{1/2}\leq \dWS(\bld{\mu}_{\kappa}^T,\bld{\mu}_0^T)\leq C_{X,T}\left(\frac{\ln(1+\kappa)}{\kappa}\right)^{1/2}.
\end{equation}
\end{theorem}

\begin{proof}
If $f\in \cC_{(0)}$, then, for $0<t<T$,  $|f(t)|\leq (1+T) |f(t)|/(1+t)$, so that
$$
\|f\|_{\cC(0,T)}=\max_{0\leq t\leq T} |f(t)|\leq (1+T)\|f\|_{(0)}
$$
and the upper bound in \eqref{eq:rate2T} follows from the upper bound in \eqref{eq:rate2}, with $C_{X,T}=(1+T)C_X$.

The lower bound, with
$$
c_{_{X,T}}=\frac{1}{2} \inf_{\kappa\geq 1}\frac{\bE\left[\max_{0\leq t \leq T} X^{\kappa}(t)\right]}{\sqrt{\ln(1+\kappa)}},
$$
follows after  repeating the corresponding steps in the proof of Theorem \ref{th:rate}.
\end{proof}

\section{Discrete Time}
\label{sec:DT}

Consider  a stationary Gaussian sequence $X=\{X_n,\ n\geq 0\},$
with $\bE X_n=0$ and $\bE X_{k+n} X_k=(1-a)a^n/(1+a),\ n,k\geq 0$, where $a\in (-1,1)$ and $a=0$ corresponds to the sequence of iid standard Gaussian random variables.
The variance of $X_n$ is chosen so that, for all $a\in (-1,1)$, the  covariance function  $R(n)=\bE X_{k+n} X_n$ of $X$ satisfies
\begin{equation*}
\label{DT-norm}
R(0)+2\sum_{n=1}^{\infty} R(n)=1.
\end{equation*}

 Using a collection   $\xi_k,\ k=0,\pm1,\pm2,\ldots$   of iid standard normal random variables,
we get the discrete-time analogs of \eqref{S-OU1} and \eqref{S-OU2}:
\begin{align}
\label{DTa1}
X_n&=(1-a)\sum_{k=-\infty}^n a^{n-k}\xi_k,\\
\label{DTa2}
X_{n+1}&=aX_n+\xi_{n+1};
\end{align}
in \eqref{DTa2}, the initial condition $X_0$ is independent of $\xi_k,\ k\geq 1,$ and  is a  normal random variable  with mean $0$ and variance $(1-a)/(1+a)$.

For $x>0$, let $\lfloor x \rfloor $ denote the largest integer that is less than or equal to $x$.
 Define  the processes $W^{\kappa}$ by
\begin{equation}
\label{Wk-dt}
W^{\kappa}(t)=\frac{1}{\sqrt{\kappa}}\sum_{n=1}^{\lfloor \kappa t \rfloor} X_n
+\frac{\kappa t-\lfloor \kappa t \rfloor}{\sqrt{\kappa}} X_{\lfloor \kappa t \rfloor+1},\ t\geq 0,\ \kappa\geq 1.
\end{equation}
The second term on the right-hand side of \eqref{Wk-dt} ensures that $W^{\kappa}$ is a continuous function of $t$.

The case $a=0$, that is, the Gaussian version of the original Donsker theorem, is of special interest; the corresponding  process $W^{\kappa}$ will be denoted by
$S_{\kappa}$:
\bel{a0}
S_{\kappa}(t)=\frac{1}{\sqrt{\kappa}}\sum_{n=1}^{\lfloor \kappa t \rfloor} \xi_n
+\frac{\kappa t-\lfloor \kappa t \rfloor}{\sqrt{\kappa}} \xi_{\lfloor \kappa t \rfloor+1},\ t\geq 0,\ \kappa\geq 1.
\ee

We have $\mathbb{P}(W^{\kappa}\in \cC_{(0)})=1$ for every $\kappa\geq 1$, and $a\in (-1,1)$,  because the ergodic theorem implies
$$
\lim_{t\to \infty} \frac{1}{\lfloor \kappa t \rfloor}\sum_{n=1}^{\lfloor \kappa t \rfloor} X_n = \bE X_0 = 0
$$
 with probability one.

Let $W=W(t),\ t\geq 0$, be a standard Brownian motion, and denote by $\bld{\mu}_0$ and $\bld{\mu}_{\kappa}$
 the measures on $\cC_{(0)}$  generated by the processes $W$ and $W^{\kappa}$. The following is the discrete-time analog of Theorem \ref{th:rate}.

\begin{theorem}
\label{th:rate-dt}
There exist positive constants  $C_a$ and $c_a$ such that, for every $\kappa\geq 1$,
\begin{equation}
\label{eq:rate-dt-a}
 c_a\left(\frac{\ln(1+\kappa)}{\kappa}\right)^{1/2} \leq  \dWS(\bld{\mu}_{\kappa},\bld{\mu}_0)\leq C_a\left(\frac{\ln(1+\kappa)}{\kappa}\right)^{1/2}.
\end{equation}
\end{theorem}

\begin{proof} The   steps are the same as in the proof of Theorem \ref{th:rate}.

Substituting  \eqref{DTa1} in \eqref{Wk-dt} and  changing the order of summation,
\bel{dt-pr1}
W^{\kappa}(t)=S_{\kappa}(t)+\frac{X^{\kappa}(t)}{\sqrt{\kappa}},
\ee
where $S_{\kappa}$ is from \eqref{a0} and
$$
X^{\kappa}(t) = \left( (a-a^{\lfloor \kappa t \rfloor})\sum_{n=-\infty}^{0}a^{-n}\xi_n -\sum_{n=1}^{\lfloor \kappa t \rfloor}
a^{\lfloor \kappa t \rfloor-n}\xi_n\right).
$$
As a result, it is enough  to establish \eqref{eq:rate-dt-a} when $a=0$:
\begin{equation}
\label{eq:rate-dt-0}
c_0\left(\frac{\ln(1+\kappa)}{\kappa}\right)^{1/2}\leq  \dWS(\bld{\mu}_{\kappa},\bld{\mu}_0)\leq C_0\left(\frac{\ln(1+\kappa)}{\kappa}\right)^{1/2}.
\end{equation}
Then, similar to the  continuous time case, we see that
$$
\bE \|X^{\kappa}\|_{(0)} \leq \bar{C}_a \sqrt{\ln(1+\kappa)},
$$
with a suitable constant $\bar{C}_a$, and then \eqref{eq:rate-dt-a} follows from \eqref{eq:rate-dt-0} with $C_a=C_0+\bar{C}_a$ and $c_a=c_0$.

 To prove \eqref{eq:rate-dt-0},  we choose the random variables $\xi_k$ in  \eqref{a0} as the increments
 of the Brownian motion $W$:
\bel{SK-emb}
\frac{\xi_n}{\sqrt{\kappa}} = W(n/\kappa)- W((n-1)/\kappa).
\ee
Then, for $(n-1)/\kappa \leq t \leq  n/\kappa$, the process $t\mapsto S_{\kappa}- W $ is a Brownian
bridge, and, for every    function $\varphi:  \cC_{(0)}\to \bR$ satisfying \eqref{Lip1},
\bel{dt-main}
|\bE \varphi(S_{\kappa})-\bE\varphi(W)|\leq  \bE\|B^{\kappa}\|_{(0)},
\ee
where  $B^{\kappa}$ is  a collection of  independent  Brownian bridges on $[(n-1)/\kappa, n/\kappa]$, $n=1,2,\ldots$.

Direct computations show that, for  $N=1,2,\ldots$,
\begin{align}
\label{dt:ub}
&  \sqrt{\kappa}\;\bE \max_{0\leq t\leq N/\kappa} |B^{\kappa}(t)|\leq 8\sqrt{{\ln (N+1)}},\\
\label{dt:lb}
&\sqrt{\kappa}\;\bE \max_{0\leq t\leq N/\kappa} B^{\kappa}(t)   \geq 0.3\sqrt{{\ln (N+1)}};
\end{align}
 the numbers $0.3$ and $8$ do not necessarily provide optimal bounds.
 Indeed, let $B=B(t),\ t\in [0,1],$ be the standard  Brownian bridge, and let
 $$
\eta=\max_{0\leq t\leq 1} B(t).
$$
 Then
 $$
\sqrt{\kappa} \;\bE \max_{0\leq t\leq N/\kappa}  |B^{\kappa}(t)| = \bE\max_{k=1,\ldots,N}|\eta_k|,
$$
where $\eta_k,\ k=1,\ldots, N$, are iid copies of  $\eta$.
Also,
$$
 \bE\max_{k=1,\ldots,N}\eta_k\leq  \bE\max_{k=1,\ldots,N}|\eta_k| \leq 2 \bE\max_{k=1,\ldots,N}\eta_k.
$$
To derive  \eqref{dt:ub}, we   repeat the arguments from the proof of  Lemma 10.1 in \cite{GaussProcLect-Lifshits} using \eqref{BB-tail}
 and   conclude that $\bE\max\limits_{k=1,\ldots,N}\eta_k\leq 4 \sqrt{\ln (N+1)}$. Similarly, for \eqref{dt:lb}, we repeat the proof of Lemma 10.2 in
  \cite{GaussProcLect-Lifshits}.

 Next, denote by $\bar{B}$ the process $B^{\kappa}$ corresponding to $\kappa=1$, that is, the collection of independent standard Brownian bridges on $[n-1,n]$, $n\geq 1$.
 Then we get the upper bound in \eqref{eq:rate-dt-0}, with
 $$
C_0=\bE\left[\sup_{t>0} \frac{|\bar{B}(t)|}{\sqrt{\ln(2+t)}}\right]< 14,
$$
by combining    \eqref{dt-main} and \eqref{dt:ub}; the upper bound on $C_0$ is from \eqref{CX-X}, because, by \eqref{I-FS},
 $C_0\leq C_X$. The lower bound in \eqref{eq:rate-dt-0}, with
$$
c_0=\frac{1}{2}\inf_{\kappa\geq 1} \frac{\bE\left[ \max_{0\leq t \leq \kappa} \bar{B}(t)\right]}{\sqrt{\ln(1+\kappa)}}> 0.2,
$$
 follows from \eqref{dt:lb} after the same arguments as in the proof of Theorem \ref{th:rate}; the lower bound on $c_0$ follows from \eqref{dt:lb}.
\end{proof}

For $T>0$, let $\cC(0,T)$ be the space of continuous functions on $[0,T]$ with the sup norm,
 and  denote by $\bld{\mu}_{0}^T$ and $\bld{\mu}_{\kappa}^T$ the measures on $\cC(0,T)$  generated by the processes $W$ and $W^{\kappa}$.
The discrete-time version of Theorem \ref{th:rateT}   is obvious. When $a=0$, and there is no continuous-time analog,
we also have the following result (cf. \cite[Proposition 2.1]{StochSim}).
\begin{proposition}
\label{prop:dt-sup}
If $\bld{\mu}_{\kappa}^T$ is the measure on $\cC(0,T)$ generated by the process $S_{\kappa}$ from \eqref{a0}, then
\bel{dt-asymptot}
\lim_{\kappa\to \infty} \sqrt{\frac{\kappa}{\ln\kappa}} \,\dWS(\bld{\mu}_{\kappa}^T,\bld{\mu}_0^T) = {\sqrt{2}}.
\ee
\end{proposition}

\begin{proof}
Using the random variables $\eta_k$ from the proof of Theorem \ref{th:rate-dt},
$$
\limsup_{\kappa\to \infty} \sqrt{\frac{ \kappa}{\ln \kappa}}\; \dWS(\bld{\mu}_{\kappa}^T,\bld{\mu}_0^T) \leq
\lim_{N\to \infty}\frac{\bE\left[\max\limits_{k=1,\ldots,N}|\eta_k|\right]}{\sqrt{\ln N}},
$$
and
$$
\liminf_{\kappa\to \infty} \sqrt{\frac{ \kappa}{\ln \kappa}} \;\dWS(\bld{\mu}_{\kappa}^T,\bld{\mu}_0^T) \geq
\lim_{N\to \infty}\frac{\bE \left[\max\limits_{k=1,\ldots,N}\eta_k\right]}{\sqrt{\ln N}}.
$$
By \eqref{BB-tail} and \cite[Theorem 1]{MaxStability-iid},
$$
\lim_{N\to \infty}\frac{\max\limits_{k=1,\ldots,N}\eta_k}{\sqrt{\ln N}}
=\lim_{N\to \infty} \frac{\max\limits_{k=1,\ldots,N}|\eta_k|}{\sqrt{\ln N}}={\sqrt{2}}
$$
with probability 1; then uniform integrability \cite[Theorem 2.1]{Pickands-max} implies \eqref{dt-asymptot}.
\end{proof}

\section{Applications}
\label{sec:ex}

Let $W$ be a standard Brownian motion and let $W^{\kappa}$ be the process from \eqref{Wk} or \eqref{Wk-dt}.
Consider    a continuous  mapping $\Psi : \cC_{(0)} \to \cC_{(0)}$. Denote by $\bld{\mu}_{0,\psi}$ and $\bld{\mu}_{\kappa,\psi}$ the measures on $\cC_{(0)}$
generated by the processes $\Psi(W)$ and $\Psi(W^{\kappa})$.
\begin{proposition}
\label{c1}
We have weak convergence $\lim_{\kappa\to \infty} \bld{\mu}_{\kappa,\psi} = \bld{\mu}_{0,\psi}$. Moreover, if there exists a number $C_{\psi}$ such that, for all
$f,g\in \cC_{(0)}$,
 \begin{equation}
 \label{CPsi}
 \|\Psi(f)-\Psi(g)\|_{(0)}\leq C_{\psi}\|f-g\|_{(0)},
 \end{equation}
  then
 \begin{equation}
 \label{DWS-m}
  \dWS(\bld{\mu}_{\kappa,\psi},\bld{\mu}_{0,\psi})\leq \bar{C}_XC_{\psi}\left(\frac{\ln(1+\kappa)}{\kappa}\right)^{1/2},
  \end{equation}
  with $\bar{C}_X=C_X$ from \eqref{CX} in continuous time and $\bar{C}_X=C_a$ from \eqref{eq:rate-dt-a} in discrete time.
  \end{proposition}

  \begin{proof} Weak convergence follows  by the continuous mapping theorem (e.g. \cite[Theorem 2.7]{Bill}).
  To establish \eqref{DWS-m}, we use either \eqref{WD-Wk} or the upper bound in \eqref{eq:rate-dt-a} and
  note that if $\varphi:  \cC_{(0)}\to \bR$ satisfies  \eqref{Lip1}, then
  $$
  |\varphi(\Psi(f))-\varphi(\Psi(g))|\leq C_{\psi}\|f-g\|_{(0)}.
  $$
  \end{proof}

  {\bf Example 1.} Given  ${\alpha}>0$,  let $Y^{\kappa}, Y$ be the solutions of
  $$
  Y^{\kappa}(t) = -{\alpha}\int_0^t Y^{\kappa}(s)\, ds + W^{\kappa}(t),\ \ Y(t)=-{\alpha}\int_0^t Y(s)\, ds + W(t),\ \ t\geq 0.
  $$
  Then $Y^{\kappa}, Y\in \cC_{(0)}$, and the  corresponding measures $\bld{\nu}_{\kappa}, \bld{\nu}$ on $\cC_{(0)}$ satisfy
  $$
  \dWS(\bld{\nu}_{\kappa},\bld{\nu})\leq 2\bar{C}_X\left(\frac{\ln(1+\kappa)}{\kappa}\right)^{1/2},\ \kappa\geq 1;
  $$
  the constant $\bar{C}_X$ is from Proposition \ref{c1}. 
  
Indeed, by  direct computation,
 $$
 Y^{\kappa}(t)=\Psi_{\alpha}(W^{\kappa})(t), \ \ Y(t)=\Psi_{\alpha}(W)(t),
 $$
 where
 \begin{equation}
 \label{Psi-a}
 \Psi_{\alpha} : f(t)\mapsto f(t)-{\alpha}\int_0^t e^{-{\alpha}(t-s)} f(s) \, ds,\ \ f \in \mathcal{C}_{(0)},
 \end{equation}
 is a linear operator. To see that $\Psi_{\alpha}$ maps $\cC_{(0)}$ to itself, note that,  for every $t>T>0$,
 \begin{align*}
\frac{|\Psi_{\alpha}(f)(t)|}{1+t}&\leq \frac{|f(t)|}{1+t}
+\frac{{\alpha}}{1+t}\int_0^t e^{-{\alpha}(t-s)}|f(s)|\, ds\\
& \leq \frac{|f(t)|}{1+t}+\frac{{\alpha}}{1+t}\int_0^T e^{-{\alpha}(t-s)}|f(s)|\, ds \\
&+ \frac{{\alpha}e^{-{\alpha}t}}{1+t}\int_T^t (1+s)e^{{\alpha}s}\, \frac{|f(s)|}{1+s}\, ds.
\end{align*}
If $\lim_{t\to \infty}|f(t)|/(1+t) = 0$, then, for every $\varepsilon>0$, we can find  $T$ so that $|f(s)|/(1+s)<\varepsilon$, $s>T$.
As a result, keeping in mind that
$$
\frac{{\alpha}}{1+t}\int_T^t (1+s) e^{{\alpha}s}\, ds \leq {\alpha}\int_0^t e^{{\alpha}s}\, ds \leq e^{{\alpha}t},
$$
we compute
$$
\limsup_{t\to\infty} \frac{|\Psi_{\alpha}(f)(t)|}{1+t} \leq \varepsilon
$$
and conclude that $\lim_{t\to\infty} |\Psi_{\alpha}(f)(t)|/(1+t)=0$.
Similarly,
 \begin{equation}
 \label{Psi-a-n}
 \|\Psi_{\alpha}(f)\|_{(0)} \leq  \|f\|_{(0)} \left(1+\alpha\int_0^{+\infty}e^{-\alpha t}\, dt \right) =
   2\|f\|_{(0)},
 \end{equation}
so that \eqref{CPsi} holds with $C_{\psi}=2$.
\hfill $\Box$

The analog of Proposition  \ref{c1} on a bounded interval is as follows.
Let $\Psi^T$ be a continuous  mapping of $\cC(0,T)$ to itself. Denote by $\bld{\mu}_{0,\psi}^T$ and $\bld{\mu}_{\kappa,\psi}^T$ the measures on $\cC(0,T)$
generated by the processes $\Psi^T(W)$ and $\Psi^T(W^{\kappa})$.

\begin{proposition}
We have weak convergence $\lim_{\kappa\to \infty} \bld{\mu}_{\kappa,\psi}^T = \bld{\mu}_{0,\psi}^T$. Moreover, if there exists a number $C_{\psi}^T$ such that, for all
$f,g\in \cC(0,T)$,
 \begin{equation}
 \label{CPsiT}
 \|\Psi^T(f)-\Psi^T(g)\|_{\cC(0,T)}\leq C_{\psi}^T\|f-g\|_{\cC(0,T)},
 \end{equation}
  then
 $$
  \dWS(\bld{\mu}_{\kappa,\psi}^T,\bld{\mu}_{0,\psi}^T)\leq (1+T)\bar{C}_{X}C_{\psi}^T\left(\frac{\ln(1+\kappa)}{\kappa}\right)^{1/2},
 $$
  with $\bar{C}_X=C_X$ from \eqref{CX} in continuous time and $\bar{C}_X=C_a$ from \eqref{eq:rate-dt-a} in discrete time.
  \end{proposition}

{\bf Example 2.} Let the function  $b=b(x), \ x\in \bR,$ satisfy
\begin{equation}
\label{bK}
|b(x)-b(y)|\leq K|x-y|,\ x,y\in \bR,
\end{equation}
and let  $Y^{\kappa}, Y$ be the solutions of
  $$
  Y^{\kappa}(t) = \int_0^t b\big(Y^{\kappa}(s)\big)\, ds + W^{\kappa}(t),\ \ Y(t)=\int_0^t b\big(Y(s)\big)\, ds + W(t),\ \ 0\leq t\leq T.
  $$
  If $\bld{\nu}_{\kappa}^T, \bld{\nu}^T$ are the corresponding measures on $\cC(0,T)$, then
  $$
  \dWS(\bld{\nu}_{\kappa}^T,\bld{\nu}^T)\leq (1+T)\bar{C}_Xe^{KT}\left(\frac{\ln(1+\kappa)}{\kappa}\right)^{1/2}.
  $$
Indeed,  for $f\in \cC(0,T)$, define  $\Psi^T(f)(t)=y(t)$ as the solution of
$$
y(t)=\int_0^t b\big(y(s)\big)\, ds + f(t),\ 0\leq t\leq T.
$$
By direct computation (e.g. \cite[Chapter 4, Lemma 1.1]{FW-LD}), we have \eqref{CPsiT} with $C_{\psi}^T = e^{KT}$.

{$  $ } \hfill $\Box$

{\bf Example 3.} Let us  combine Examples 1 and 2. Take a positive number ${\alpha}$ and a function $b=b(x)$ satisfying \eqref{bK}, and
 let  $Y^{\kappa}, Y$ be the solutions of
  \begin{align*}
  Y^{\kappa}(t)& =  -{\alpha}\int_0^t Y^{\kappa}(s)\, ds+\int_0^t b\big(Y^{\kappa}(s)\big)\, ds + W^{\kappa}(t),\\
   Y(t)&= -{\alpha}\int_0^t Y(s)\, ds+\int_0^t b\big(Y(s)\big)\, ds + W(t),\ \ t\geq 0.
  \end{align*}
If ${\alpha}>K$, then $Y^{\kappa}, Y \in \cC_{(0)}$ and, for the corresponding measures $\bld{\nu}_{\kappa}, \bld{\nu}$,
  \begin{equation}
  \label{dws-OU-b}
  \dWS(\bld{\nu}_{\kappa},\bld{\nu})\leq \frac{2{\alpha}\bar{C}_X}{{\alpha}-K}\left(\frac{\ln(1+\kappa)}{\kappa}\right)^{1/2}.
  \end{equation}
  Indeed, for $f\in \cC_{(0)}$, define  $\Psi(f)(t)=y(t)$ as the solution of
$$
y(t)=-{\alpha}\int_0^t y(s)\, ds + \int_0^t b\big(y(s)\big)\, ds + f(t),\ t\geq 0.
$$
Using variation of parameters formula and \eqref{Psi-a},
$$
\Psi(f)(t)=\int_0^t e^{-{\alpha}(t-s)} b\big(\Psi(f)(s)\big)\, ds + \Psi_a(f)(t).
$$
Then, similar to Example 1, we conclude that $\Psi$ maps $\cC_{(0)}$ to itself.
In particular, using \eqref{bK} and \eqref{Psi-a-n},
$$
\Psi(f)(t)-\Psi(g)(t) =\int_0^t e^{-{\alpha}(t-s)} \Big( b\big(\Psi(f)(s)\big)-b\big(\Psi(g)(s)\big)\Big)\, ds + \Psi_{\alpha}(f-g)(t)
$$
so that
$$
\|\Psi(f)-\Psi(g)\|_{(0)} \leq K \|\Psi(f)-\Psi(g)\|_{(0)} \int_0^{\infty} e^{-as}\, ds + 2\|f-g\|_{(0)}.
$$
As a result, if $\alpha>K$, then
$$
\|\Psi(f)-\Psi(g)\|_{(0)}\leq  \frac{2{\alpha}}{{\alpha}-K}\|f-g\|_{(0)},
$$
and \eqref{dws-OU-b} follows from \eqref{DWS-m}. \hfill $\Box$

\section{Concluding Remarks}
\label{sec:sum}

A proof of the functional Central Limit Theorem for  processes of the type \eqref{Wk} or \eqref{Wk-dt} usually includes the following steps:
\begin{enumerate}
\item A Gordin-type decomposition \cite{Gordin}, when $W^{\kappa}$ is written as a sum of a martingale and an a ``small'' correction;
\item A coupling argument, when  $W^{\kappa}$ is constructed on the same probability space as $W$;
\item A Skorokhod embedding for the martingale component of $W^{\kappa}$.
\end{enumerate}
Each step leads to an approximation error; in particular, \cite{KMT1, KMT2} developed a systematic procedure, now known as the KMT approximation,  to minimize the error due to the Skorokhod embedding.
 When the underlying processes are Gaussian, some  of the approximation errors are not present.

 In continuous time,
the first two steps are the equality \eqref{approxW-1}. There is no need for Skorokhod embedding because the martingale component is the Brownian motion.
 In discrete time,
the first step is  the equality \eqref{dt-pr1}, whereas \eqref{SK-emb}  represents coupling and the Skorokhod embedding.
 For  convergence  in the space of continuous functions, the $\sqrt{\ln \kappa}$ correction to the classical rate $1/\sqrt{\kappa}$ comes
 from the growth  of the maximum of iid standard Gaussian random variables.

Keeping in mind that  rate of convergence in the functional CLT can depend both on the
underlying functional space and on the   distance between the
measures on that space,  the rate  $1/\sqrt{\kappa}$ is possible to achieve.  For example,
 by considering $W$ and $S_{\kappa}$ [from \eqref{a0}]  as processes in $L_1(0,T)$, as opposed to $\cC(0,T)$,
direct computations \cite[Proposition 2.1]{StochSim} yield
$$
\bE \int_0^1 |S_{\kappa}(t)-W(t)|\, dt=\frac{1}{\sqrt{\kappa}}\int_0^1 \bE |B(t)|\, dt = \sqrt{\frac{2}{\pi\,\kappa}} \int_0^1\sqrt{t(1-t)}\, dt =
 \sqrt{\frac{\pi}{32\,\kappa}},
 $$
 that is,  the Wasserstein-1 distance between $S_{\kappa}$ and $W$ in $L_1(0,T)$   is of order $1/\sqrt{\kappa}$; see also \cite[Remark 1]{Barbour-diff-stein}.

 Given the variety of function spaces that can support $W$ and $W^{\kappa}$, as well as the variety of   ways to measure the distance between the corresponding
 probability distributions \cite{GibbsSu}, identifying all situations with a sharp $1/\sqrt{\kappa}$ bound becomes an interesting challenge.
  For  $\dWS(W,W^{\kappa})$ in the space of continuous functions with
 the sup norm,  there is strong evidence that convergence cannot
  be faster than $\sqrt{\ln \kappa/\kappa}$:
 the results of this paper demonstrate it
   in the Gaussian case, and, by \cite[Corollary 4.4]{RW-BM-rate},
   the simple symmetric random walk cannot beat this rate either.

% \bibliographystyle{amsplain}
%\bibliography{WCERef-1}

\def\cprime{$'$}
\providecommand{\bysame}{\leavevmode\hbox to3em{\hrulefill}\thinspace}
\providecommand{\MR}{\relax\ifhmode\unskip\space\fi MR }
% \MRhref is called by the amsart/book/proc definition of \MR.
\providecommand{\MRhref}[2]{%
  \href{http://www.ams.org/mathscinet-getitem?mr=#1}{#2}
}
\providecommand{\href}[2]{#2}

\end{document}